%% file: vN_fp.tex
\documentclass{amsart}

\usepackage{amssymb}
\usepackage{mathrsfs}
\usepackage[all]{xy}

\input{macros}

\title[von Neumann-Day problem]
{A finitely presented group of piecewise projective homeomorphisms}



\thanks{
This research was supported in part by
NSF grant DMS--1262019.
The second author would like to thank Alain Valette for a number of helpful conversations
concerning Monod's example in \cite{Monod} which took place while the author
was visiting  Universit\'e de Neuch\^{a}tel in January of 2013.}

\author{Yash Lodha}

\author{Justin Tatch Moore}

\address{Department of Mathematics \\ Cornell University\\
Ithaca, NY 14853--4201 \\ USA}

\email{{\tt yl763@cornell.edu}}

\dedicatory{This paper is dedicated to the memory of William Thurston (1946--2012).} 
\begin{document}

\begin{abstract}
In this article we will describe a finitely presented subgroup of Monod's group
of piecewise projective homeomorphisms of $\Rbb$.
This in particular provides a new example of a finitely presented group which is
nonamenable and yet does not contain a nonabelian free subgroup.
It is in fact the first such example which is torsion free.
We will also develop a means for representing the elements of the group by
labeled tree diagrams in a manner which closely parallels Richard Thompson's group $F$.
\end{abstract}

\maketitle

\section{Introduction}

The notion of an \emph{amenable group} was introduced by von Neumann
as an abstract means for preventing the existence
of paradoxical decompositions of the group:
a discrete group is \emph{amenable}
if it admits a finitely additive translation invariant probability measure.
At the heart of Banach and Tarski's paradoxical decomposition of the sphere
is the existence of a paradoxical decomposition of the free group
on two generators.
Since subgroups of amenable groups are easily seen to be amenable, it is natural to
ask whether every nonamenable group contains a free group on two generators.
Day was the first to pose this problem in print \cite{Day}, where he
attributed it to John von Neumann.

In 1980, Ol'shanskii solved the von Neumann-Day problem by producing a counterexample \cite{Olsh}.
Soon after, Adyan showed that certain Burnside groups are also counterexamples \cite{Adyan1}\cite{Adyan2}.
These examples are not finitely presented and the restriction of the
von Neumann-Day problem to the class of finitely presented groups remained open
until Ol'shanskii and Sapir constructed an example in 2003 \cite{OlshSap}.
Shortly after, Ivanov published another finitely presented counterexample \cite{Ivanov},
that is somewhat simpler but similar in spirit to the Ol'shanskii-Sapir example.
Both examples were produced by elaborate inductive constructions and are difficult to analyze.
It is also interesting to note that both of these examples are based on the constructions
of nonamenable torsion groups (they are torsion-by-cyclic)
and in particular are far from being torsion free.



In his recent article \cite{Monod}, Monod produced a new family of counterexamples 
to the von Neumann-Day problem.
They are all subgroups of the group $H$ consisting of all piecewise
projective transformations of the real projective line which fix the point
at infinity.
Monod demonstrated that $H$ does not contain nonabelian free subgroups
by adapting the method of Brin and Squier \cite{BrinSq}.

In this article, we will isolate a
finitely presented nonamenable subgroup of Monod's group $H$.
To our knowledge, this provides the first finitely presented
torsion free counterexample to the von Neumann-Day problem.
Moreover, our presentations for this group are very explicit:
it has a presentation with 3 generators and 9 relations as well as a natural
infinite presentation.
The group is generated by $a(t) = t+1$ together with the following
two homeomorphisms of $\Rbb$:
\[
b(t)=
\begin{cases}
 t&\text{ if }t\leq 0\\
 \frac{t}{1-t}&\text{ if }0\leq t\leq \frac{1}{2}\\
 3-\frac{1}{t}&\text{ if }\frac{1}{2}\leq t\leq 1\\
 t+1&\text{ if }1\leq t\\
\end{cases}
\qquad
c(t)=
\begin{cases}
 \frac{2t}{1+t}&\text{ if }0\leq t\leq 1\\
t&\text{ otherwise}\\
\end{cases}
\]
The main result of this paper is the following.
\begin{thm} \label{main}
The group $G_0$ generated by the functions $a(t)$, $b(t)$, and $c(t)$
is nonamenable and finitely presented.
\end{thm}
Since it is a subgroup of $H$, $G_0$ does not contain a nonabelian free subgroup.
We claim no originality in our proof that $G_0$ is nonamenable; this is a routine modification
of the methods of \cite{Monod} which in turn relies on \cite{CarGhys} \cite{Connes}.
 
It is interesting to note that, by an unpublished result of Thurston, $a(t)$ and
$b(t)$ generate the subgroup $P(\Zbb) \leq H$, consisting of those homeomorphisms which
are $C^1$ and piecewise $\PSL_2(\Zbb)$, which he moreover showed is
a copy of Richard Thompson's group $F$.
In fact the methods of \cite{Monod} easily show that 
$t \mapsto t + 1/2$ and $b(t)$ generate a nonamenable group, although at present it is
unclear whether this group is finitely presented.
While there are strong parallels between the group $\langle a,b,c\rangle$ and $F$,
neither \cite{Monod} nor the present article seems to shed any light on whether $F$ is
amenable.

The paper is organized as follows.
Section \ref{prelim} contains a review of some of the preliminaries we will need later 
in the paper.
Both an infinite and a finite presentation are described in Section \ref{presentation}
and it is demonstrated there how to prove that the finite presentation generates the infinite presentation.
Tree diagrams for elements of the group are developed in Section \ref{tree_diag}.
Finally, Section \ref{sufficient} contains a proof that the relations isolated in Section \ref{presentation}
suffice to give a presentation for $\Seq{a,b,c}$.

\section{Preliminaries}

\label{prelim}
Our analysis of the group $G_0$ will closely parallel the
now well-established analysis of Thompson's group $F$.
We direct the reader to the standard reference \cite{Belk} for
the properties of $F$; additional information can be found in \cite{Cannon}.
We shall mostly follow the notation and conventions
of \cite{Belk} \cite{BrinSq}.

We will take $\mathbb{N}$ to include $0$; in particular all
counting will start at $0$.
Let $2^\Nbb$ denote the collection of all infinite binary sequences and
let $2^{<\Nbb}$ denote the collection of all finite binary sequences.
If $i \in \Nbb$ and $u$ is a binary sequence of length at least $i$, we will let 
$u\restriction i$ denote the initial part of $u$ of length $i$.
If $s$ and $t$ are finite binary sequences, then we will write $s \subseteq t$
if $s$ is an initial segment of $t$ and $s \subset t$ if $s$ is a proper
initial segment of $t$.
If neither $s \subseteq t$ nor $t \subseteq s$, then we will say that
$s$ and $t$ are \emph{incompatible}.
The set $2^{<\Nbb}$ is equipped with a lexicographic order defined
by $s <_\lex t$ if $t \subset s$ or $s$ and $t$ are incompatible
and $s(i) < t(i)$ where $i$ is minimal such that $s(i) \ne t(i)$.

If $s$ and $t$ are two sequences ($s$ is finite but $t$ may be infinite),
then $s \cat t$ will be used to denote the concatenation of $s$ and $t$.
In some circumstances, $\cat$ will be omitted;
for instance we will often write $s \seq{01}$ instead of $s \cat \seq{01}$.
If $\xi$ and $\eta$ are infinite sequences, then we will say that
$\xi$ and $\eta$ are \emph{tail equivalent} if there are $s$, $t$, and $\zeta$ such that
$\xi = s \cat \zeta$ and $\eta = t \cat \zeta$.
Given an infinite sequence $s$, $\tilde{s}$ is the sequence 
satisfying $\tilde{s}(i)=\seq{0}$ if $s(i)=\seq{1}$ and $\tilde{s}(i)=\seq{1}$ if
$s(i)=\seq{0}$.
The constant sequences $\seq{000}...,\seq{111}...$ are denoted by $\seq{\bar 0},\seq{\bar 1}$
respectively.

Let $\Tcal$ denote the collection of all finite rooted ordered binary trees. 
More concretely, we view elements $T$ of $\Tcal$ as \emph{prefix} sets --- those
sets $T$ of finite binary sequences with the property that every infinite 
binary sequence has a unique initial segment in $T$.
Observe that, for each $m$, there are only finitely many elements of $\Tcal$
with $m$ elements.
There is also a natural ordering on $\Tcal$, which we will refer to as \emph{dominance}:
if every element of $S$ has an extension in $T$, then we say that $S$ is dominated by $T$.
Notice that if $S$ is dominated by $T$, then $|S| \leq |T|$.
If $A$ is a finite set of binary sequences, then there is a unique minimal
element $T$ of $\Tcal$ (with respect to the order of dominance) such that
every element of $A$ has an extension in $T$.

A \emph{tree diagram} is a pair $(L,R)$ of elements of $\Tcal$
with the property that $|L|=|R|$.
A tree diagram describes a map of infinite binary sequences
as follows:
\[
s_i\cat \xi \mapsto t_i\cat \xi
\]
where $s_i$ and $t_i$ are the $i$th elements of $L$ and $R$ respectively
and $\xi$ is any binary sequence.
The collection of all such functions from $2^\Nbb$ to $2^\Nbb$ defined
in this way is \emph{Thompson's group $F$}.
This map is also defined on any finite binary sequence $u$ such that
$u$ has a prefix in $L$.
We will follow \cite{BrinSq} and write $s.f$ for the result of applying
an automorphism $f$ to the input $s$.
The operation of $F$ is therefore defined as $s.(fg) = (s.f).g$.
Thompson's group $F$ is generated by the following functions.
\[
\xi.a
 =
\begin{cases}
\seq{0}\eta & \textrm{ if } \xi = \seq{00} \eta \\
\seq{10}\eta & \textrm{ if } \xi = \seq{01} \eta \\
\seq{11}\eta & \textrm{ if } \xi = \seq{1} \eta \\
\end{cases}
\qquad \qquad
\xi.b =
\begin{cases}
\seq{0}\eta & \textrm{ if } \xi = \seq{0}\eta \\
\seq{10}\eta & \textrm{ if } \xi = \seq{100} \eta \\
\seq{110}\eta & \textrm{ if } \xi = \seq{101} \eta \\
\seq{111}\eta & \textrm{ if } \xi = \seq{11} \eta \\
\end{cases}
\]

Recall that the \emph{real projective line} is the set of all lines in $\Rbb^2$ which pass
through the origin.
Such lines can naturally be identified with elements of $\Rbb \cup \{\infty\}$ via the 
$x$-coordinate of
their intersection with the line $y=1$. 
In this article it will be useful to represent points on the
real projective line by binary sequences derived from their continued fractions expansion. 
Define a map $\Phi: 2^\Nbb \to \Rbb \cup \{\infty\}$ as follows.
First define $\phi:2^\Nbb \to [0,\infty]$ by
\[
\phi(\seq{0} \xi) = \frac{1}{1+\frac{1}{\phi(\xi)}} \qquad\qquad \phi(\seq{1} \xi) = 1 + \phi(\xi) 
\]
and set
\[
\Phi(\seq{0} \xi) = - \phi(\tilde{\xi}) \qquad\qquad \Phi(\seq{1} \xi) = \phi(\xi).
\]
This function is one-to-one except at $\xi$ which are eventually constant
(i.e. $r_k = \infty$ for some $k$).
On sequences which are eventually constant, the map is two-to-one:
$\Phi(s \seq{0 \bar 1}) = \Phi(s \seq{1 \bar 0})$ and $\Phi(\seq{\bar 0}) = \Phi(\seq{\bar 1}) = \infty$.

In the mid 1970s,
Thurston observed that the functions $a$ and $b$ from the introduction become the
generators $a$ and $b$ for Thompson's group $F$ defined above when ``conjugated'' by $\Phi$.
Moreover, the elements of $F$ correspond exactly to those homeomorphisms $f$
of $\Rbb$ which are piecewise $\PSL_2(\Zbb)$ and which have continuous derivatives.
We will generally take the viewpoint that $\Phi$
provides just another way of describing the real projective line, just as decimal expansions allow us
to describe real numbers.
In particular, we will regard the definitions of $a$ and $b$ in the introduction and the definitions given
above in terms of sequences as being two ways of describing the \emph{same functions}.

Since we will be proving that a group is finitely presented, it will be necessary to
deal with formal words over formal alphabets.
If $G$ is a group and $A$ is a subset of $G$, an $A$-word
is a finite sequence of elements of the set $A \times (\Zbb \setminus \{0\})$.
We typically denote a pair $(a,n)$ as $a^n$, but we emphasize here that it is formally
distinct from the group element $a^n$.
The \emph{word length} of an $A$-word is the sum of the absolute values of the exponents which occur
in it.

In order to prove the nonamenability of $G_0$, we will need to employ 
Zimmer's theory of amenable equivalence relations.
Let $X$ be a Polish space and let $E \subseteq X^2$ be an equivalence relation
which is Borel and which has countable equivalence classes.
$E$ is \emph{$\mu$-amenable} if, after discarding a $\mu$-measure $0$ set,
$E$ is the orbit equivalence relation of an action of $\Zbb$.
(This is not the standard definition, but it is equivalent by \cite{Connes}.)
We will need the following two results.

\begin{thm} \label{amen_OE} \cite{zimmer}
If $\Gamma$ is a countable amenable group acting by Borel automorphisms on a Polish space $X$
and $\mu$ is any $\sigma$-finite Borel measure on $X$, then
the orbit equivalence relation is $\mu$-amenable.
\end{thm}

\begin{thm} \label{PSL2_nonamen} \cite{CarGhys} (see also the discussion in \cite{Monod})
If $\Gamma$ is a countable dense subgroup of $\PSL_2(\Rbb)$, then the action of $\Gamma$
on the real projective line induces an orbit equivalence relation which is not amenable
with respect to Lebesgue measure. 
\end{thm}
\noindent
We refer the reader to \cite{KechMiller} for further information
on amenable equivalence relations. 

We will conclude this section by sketching a proof that the group $G_0$ from
the introduction is nonamenable.
Let $K$ denote the subgroup of $\PSL_2(\Rbb)$ generated by
the matrices
\[
\left( \begin{array}{ccc}
1 & 1\\
0 & 1 \\ \end{array} \right)
\quad
\left( \begin{array}{ccc}
0 & 1\\
-1 & 0 \\ \end{array} \right)
\quad
\left( \begin{array}{ccc}
\sqrt{2} & 0\\
0 & \frac{1}{\sqrt{2}} \\ \end{array} \right).
\]
Viewed as fractional linear transformations, $K$ is generated
by $t+1$, $2t$, and $-1/t$.
Since $K$ contains $\PSL_2(\Zbb)$ as a proper subgroup, it is dense in
$\PSL_2(\Rbb)$ and hence by Theorem \ref{PSL2_nonamen}, the orbit equivalence
relation of its action on the real projective line
is not amenable with respect to Lebesgue measure.
By Theorem \ref{amen_OE}, it is sufficient to show that $G_0$ induces the same
orbit equivalence relation on $\Rbb \setminus \Qbb$. 
To see this, it can be verified that the element
$bca^{-1}c^{-1}a$ coincides with $t \mapsto 2t$ on the interval $[0,1]$.
From this and the identity $2(r-n) + 2n = 2r$ it follows that
the orbits of $G_0$ include the orbits of the action of
$\langle t \mapsto t+1, t \mapsto 2t\rangle$ on $\Rbb \setminus \Qbb$.
Finally $aba$ and $ba^{-3}$ coincide with $t \mapsto -1/t$ on the intervals $[-1,-1/2]$
and $[1/2,1]$ respectively.
The assertion about orbits now follows from the fact that for any $r \in \Rbb$,
there is a $n \in \Zbb$ such that $2^n r$ is in $[-1,-1/2] \cup [1/2,1]$ and
$(2^n)(-1/(2^n r)) = -1/r$.

\section{The presentations}

\label{presentation}

In this section we will describe both a finite and infinite presentation of the group
$G_0$ defined in the introduction.
We start with the following two primitive functions defined on binary sequences:
\[
\xi.x
 =
\begin{cases}
\seq{0}\eta & \textrm{ if } \xi = \seq{00} \eta \\
\seq{10}\eta & \textrm{ if } \xi = \seq{01} \eta \\
\seq{11}\eta & \textrm{ if } \xi = \seq{1} \eta \\
\end{cases}
\qquad
\xi.y =
\begin{cases}
\seq{0}(\eta.y) & \textrm{ if } \xi = \seq{00} \eta \\
\seq{10}(\eta.y^{-1}) & \textrm{ if } \xi = \seq{01} \eta \\
\seq{11}(\eta.y) & \textrm{ if } \xi = \seq{1} \eta \\
\end{cases}
\]
(The function $x$ is nothing but the function $a$ described in the previous section.)
From these functions, we define families of functions $x_s$ $(s \in 2^{<\Nbb})$ and $y_s$ $(s \in 2^{<\Nbb})$
which act just as $x$ and $y$, but localized to those binary sequences which extend $s$.
\[
\xi.x_s
 =
\begin{cases}
s \cat (\eta.x) & \textrm{ if } \xi = s \cat \eta \\
\xi & \textrm{otherwise}
\end{cases}
\qquad
\xi.y_s
 =
\begin{cases}
s \cat (\eta.y) & \textrm{ if } \xi = s \cat \eta \\
\xi & \textrm{otherwise}
\end{cases}
\]
If $s$ is the empty-string, it will be omitted as a subscript.
The relationship between these functions and the functions $a$, $b$, and $c$ of the introduction is
expressed by the following proposition.

\begin{prop}
For all $\xi$ in $2^\Nbb$,
$\phi(\xi.y) = 2 \phi(\xi)$
and
\[
\Phi(\xi).a = \Phi(\xi.x) \qquad \Phi(\xi).b = \Phi(\xi.x_{\seq{1}}) \qquad \Phi(\xi).c = \Phi(\xi.y_{\seq{10}}).
\]
\end{prop}

\begin{remark}
The effects of doubling on continued fractions was first worked out by Hurwitz \cite{Hurwitz}.
Raney introduced transducers
for making calculations such as these in \cite{cont_frac_aut}.
\end{remark}

\begin{proof}
We will only prove the identities $\phi(\xi.y) = 2 \phi(\xi)$ and $\Phi(\xi).c = \Phi(\xi.y_{\seq{10}})$;
the remaining verifications are similar and left to the reader.
We will first verify the identity $\phi(\xi.y) = 2 \phi(\xi)$.
Observe that, since $\phi$ and $y$ are continuous, it suffices to verify this equality
for sequences which are eventually constant.
The proof is now by induction on the minimum digit beyond which the sequence is constant.
For the base case we have:
\[
\phi(\seq{\bar 0}.y) = \phi(\seq{\bar 0}) = 0 = 2\cdot 0  \qquad  \phi(\seq{\bar 1}.y) = \phi(\seq{\bar 1}) = \infty = 2 \cdot \infty.
\]
In the inductive step, we have three cases:
\[
\phi(\seq{00} \xi . y) = \phi(\seq{0}(\xi . y)) = \frac{1}{1 + \frac{1}{\phi(\xi . y)}}
=\frac{1}{1 + \frac{1}{2 \phi(\xi)}} =
\frac{2}{2 + \frac{1}{\phi(\xi)}} = 2 \phi(\seq{00} \xi)
\]
\[
\phi(\seq{01} \xi . y) = \phi(\seq{10}(\xi . y^{-1})) = 1 + \frac{1}{1 + \frac{1}{\phi(\xi . y^{-1})}}
= 1 + \frac{1}{1 + \frac{2}{\phi(\xi)}} = \frac{2}{1 + \frac{1}{1 + \phi(\xi)}} = 2 \phi(\seq{01} \xi)
\]
\[
\phi(\seq{1} \xi . y) = \phi(\seq{11}(\xi . y)) = 2 + \phi(\xi . y) = 2 + 2\phi(\xi) = 2(1 + \phi(\xi)) = 2\phi(\seq{1}\xi).
\]

Next we turn to the verification of $\Phi(\xi).c = \Phi(\xi.y_{\seq{10}})$.
Observe that if $\xi$ does not extend $\seq{10}$, then $\Phi(\xi)$ is outside the interval $(0,1)$ and we have
$\Phi(\xi).c = \Phi(\xi) = \Phi(\xi.y_{\seq{10}})$.
The remaining case follows from the identity we have already established, noting that $\frac{2t}{t+1} = \frac{2}{1+1/t}$:
\[
\Phi(\seq{10}\xi).c = \left( \frac{1}{1+\frac{1}{\phi(\xi)}} \right).c =
\frac{2}{2 + \frac{1}{\phi(\xi)}} = \frac{1}{1 + \frac{1}{2\phi(\xi)}} = \frac{1}{1+\frac{1}{\phi(\xi.y)}} = \Phi(\seq{10}\xi . y_{\seq{10}}).
\]
\end{proof}
From this point forward, we will identify $a$, $b$, and $c$ with $x$, $x_\seq{1}$, and $y_{\seq{10}}$,
respectively and suppress all mention of $\Phi$.

We now return to our discussion of the generators.
It is straightforward to verify that the following relations are satisfied by these elements,
where $s$ and $t$ are finite binary sequences:
\begin{enumerate}

\item \label{pent} $x_s^2 = x_{s\seq{0}} x_s x_{s\seq{1}}$;

\item \label{xx} if $t . x_s$ is defined, then
$x_t x_s = x_s x_{t . x_s}$;

\item \label{xy} if $t . x_s$ is defined, then
$y_t x_s = x_s y_{t . x_s}$;

\item \label{yy} if $s$ and $t$ are incompatible, then $y_s y_t = y_t y_s$;

\item \label{rw} $y_s = x_s y_{s \seq{0}} y_{s \seq{10}}^{-1} y_{s \seq{11}}$.

\end{enumerate}
We will refer to these relations collectively as $R$.
The first two groups of relations are known to give a presentation for $F$:
the function $x_{\seq{1}^n}$ corresponds to the $n$th generator in the
standard infinite presentation of $F$.
We will use $F$ to denote the group generated by $\{x_s : s \in 2^{<\Nbb}\}$.

Notice that any $y_s$ is conjugate by an element of $F$ to exactly one of $y$,
$y_{\seq{0}}$, $y_{\seq{1}}$, or $y_{\seq{10}}$.
Define $X = \{x_s : s \in 2^{<\Nbb}\}$, $Y = \{y_s : s \in 2^{<\Nbb}\}$, and $Y_0$ to be the set of all
$y_s$ such that $s$ is not a constant binary sequence.
Observe that $Y_0$ consists of those elements of
$Y$ which are conjugate to $y_{\seq{10}}$ by an element of $F$.
The group $G_0$ defined in the introduction is therefore generated by the (redundant) generating set
$S_0 = X \cup Y_0$.

Let $R_0$ be those relations in $R$ which only refer to generators in $S_0$ and
let $G$ be the group generated by $S = X \cup Y$.
The rest of the paper will focus on proving the following theorem.

\begin{thm}
The relations $R$ give a presentation for $G$ and the relations $R_0$ give a presentation
for $G_0$.
Moreover, $G$ and $G_0$ admit finite presentations.
\end{thm}

In the remainder of this section, we will prove that 
$G$ and $G_0$ are finitely presented assuming that $R$ and $R_0$ give presentations for these groups.
The finite generating sets for these groups are $\{x,x_{\seq{1}},y_{\seq{0}},y_{\seq{1}},y_{\seq{10}}\}$ and
$\{a,b,c\} = \{x,x_{\seq{1}},y_{\seq{10}}\}$, respectively.
Before proceeding, it will be necessary to define the other generators as words in terms
of these generators; these definitions will be compatible with equalities which hold in $G$.
We begin by declaring
\[
y = x y_{\seq{0}} y_{\seq{10}}^{-1} y_{\seq{11}} \qquad x_{\seq{0}} = x^2 x_{\seq{1}}^{-1} x^{-1}
\qquad
x_{\seq{10}} = x_{\seq{1}}^2 x^{-1} x_{\seq{1}}^{-1} x x_{\seq{1}}^{-1}.
\]
Observe that $\seq{0}.x^{-n} = \seq{0}^{n+1}$ and $\seq{1} . x^n = \seq{1}^{n+1}$ and
set
\[
x_{\seq{0}^{n+1}} = x^{n} x_{\seq{0}} x^{-n} \qquad x_{\seq{1}^{n+1}} = x^{-n} x_{\seq{1}} x^n
\]
\[
y_{\seq{0}^{n+1}} = x^n y_{\seq{0}} x^{-n} \qquad y_{\seq{1}^{n+1}} = x^{-n} y_{\seq{1}} x^n.
\]
If $s \in 2^{<\Nbb}$ is nonconstant, fix a word $f_s$ in $\{x,x_{\seq{1}}\}$ such that
$\seq{10}.f_s = s$ and define
\[
x_s = f_s^{-1} x_{\seq{10}} f_s \qquad
y_s = f_s^{-1} y_{\seq{10}} f_s.
\]
Next we note the following two standard properties of $F$.

\begin{prop}\label{relF}
If $g$ is any element of $F$ and $s$ is a finite binary sequence
such that $s.g$ is defined, then $x_s g = g x_{s.g}$.
In particular if $g$, $x_s$ and $x_{s.g}$ are expressed as words in $\{x,x_{\seq{1}}\}$,
then the above equality is derivable from the relations in (\ref{pent}) and (\ref{xx}) above.
\end{prop}

\begin{prop}\label{incompatible}
If $u <_\lex v$ are incompatible binary sequences,
then there is a $g$ in $F$ and
$s <_\lex t$ each of length at most $3$
such that $s.g = u$ and $t.g = v$.
\end{prop}

From these facts it follows that every relation in (\ref{yy}) is conjugate via an element of $F$ to a relation
in (\ref{yy}) indexed by sequences of length at most $3$.
The relations in (\ref{rw}) are conjugate via elements of $F$ to a relation
$y_s = x_s y_{s \seq{0}} y_{s \seq{10}}^{-1} y_{s \seq{11}}$ where $s\in \{\seq{0},\seq{10},\seq{1}\}$.
The relations in (\ref{xy}) can be expressed as $y_s g=g y_s$
where $s\in \{\seq{0},\seq{10},\seq{1}\}$ and $g\in F$ such that $s.g=s$.
These can be derived from relations $y_sx_t=x_ty_s$ where $s.x_t=s$ and
$s,t$ are binary sequences of length at most $3$.
In particular $G$ and $G_0$ are finitely presented.
In the case of $G_0$, one can check that the following list of $9$ relations actually suffice:
\[
x_{\seq{1}} x^{-2} x_{\seq{1}} x =  x^{-1} x_{\seq{1}} x x_{\seq{1}} x^{-1}
\qquad 
x_{\seq{1}} x^{-3} x_{\seq{1}} x^2 = x^{-2} x_{\seq{1}} x^2 x_{\seq{1}} x^{-1} 
\]
\[
y_{\seq{10}} x_{\seq{0}}   = x_{\seq{0}}   y_{\seq{10}} \qquad
y_{\seq{10}} x_{\seq{01}}  = x_{\seq{01}}  y_{\seq{10}} 
\]
\[
y_{\seq{10}} x_{\seq{11}}  = x_{\seq{11}}  y_{\seq{10}} \qquad
y_{\seq{10}} x_{\seq{111}} = x_{\seq{111}} y_{\seq{10}} 
\]
\[ 
y_{\seq{01}} y_{\seq{10}} = y_{\seq{10}} y_{\seq{01}} \qquad
y_{\seq{001}} y_{\seq{10}} = y_{\seq{10}} y_{\seq{001}}
\]
\[
y_{\seq{10}}=x_{\seq{10}}y_{\seq{100}}y^{-1}_{\seq{1010}}y_{\seq{1011}}.
\]
(Notice that all of the above relations except the last assert that
a pair of elements of the group commute.
In each case this is because they are supported on disjoint sets,
where the support $g$ is the set of $x$ such that $x.g \ne x$.)
When expressed in terms of the original generators, these become:
\[
b a^{-2} b a =  a^{-1} b a b a^{-1}
\qquad 
b a^{-1} a^{-2} b a^2 = a^{-2} b a^2 b a^{-1} 
\]
\[
c a^2 b^{-1} a^{-1} =  a^2 b^{-1} a^{-1} c \qquad
c b^2 a^{-1} b a b  =  b^2 a^{-1} b a b c 
\]
\[
c a^{-1} b a  = a^{-1} b a c \qquad
c a^{-2} b a^2 = a^{-2} b a^2 c 
\]
\[ 
c a c a^{-1} = a c a^{-1} c \qquad
c a^2 c a^{-2} = a^2 c a^{-2} c
\]
\[
c = b^2a^{-1}b^{-1}acb^{-2}ab^{-1}c^{-1}ba^{-1}bab^{-1}ab^{-1}cba^{-1}ba^{-1}.
\]

\section{Tree diagrams}

\label{tree_diag}

Before proceeding further, we will pause to describe how the elements of 
$\Seq{a,b,c}$ can be described
in terms of tree diagrams, similar to those associated to Thompson's group $F$.
This section is not essential for understanding the proof
of Theorem \ref{main}
in the subsequent section, although the reader may find the material here
is useful in visualizing what is happening in the main proofs.

Let $\tilde \Tcal$ denote the collection of all finite sets $S$ of reduced words in the alphabet
$\{\seq{0},\seq{1},y,y^{-1}\}$ which satisfy the following properties:
\begin{itemize}

\item $S$ is nonempty;

\item the result of deleting all occurrences of $y$ and $y^{-1}$ in members of $S$ defines a bijective
map between $S$ and an element of $\Tcal$;

\item if $u y^n$ is a prefix of some element of $S$, then any element of $S$ which has $u$ as a prefix, also
has $u y^n$ as a prefix;

\end{itemize}
Elements of $\tilde \Tcal$ be be visualized as follows.
Let $S$ be in $\tilde \Tcal$ and $T$ is the result of removing the occurrences of $y$ and $y^{-1}$ from
elements of $S$.
We can think of $T$ as defining a rooted ordered binary tree, whose vertexes correspond to the prefixes of elements of $T$.
The elements of $S$ can be specified by an assignment of an integer to each vertex of $T$.
For instance if $S = \{\seq{0},\seq{1yy0y^{-1}},\seq{1yy1}\}$, then the associated labeled tree is:
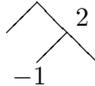
\[
\xy
(0,0);(4,4)**@{-};(12,-4)**@{-};(8,0); (4,-4)**@{-}; (10,2)*{2}; (3,-6)*{-1};
\endxy
\]
(here and below unspecified labels are 0).

A \emph{labeled tree diagram} is a pair $S \to T$ of elements of $\tilde \Tcal$ such that
$S$ and $T$ have the same number of vertexes.
The key point is now to define the appropriate notion of equivalence of tree diagrams.
First we define a notion of equivalence on $\tilde \Tcal$.
Two (possibly infinite) words in the alphabet $\{\seq{0},\seq{1},y,y^{-1}\}$ are equivalent if one can be converted
into the other by the following substitutions:
\[
y\seq{00} \Leftrightarrow \seq{0}y \qquad y\seq{01} \Leftrightarrow \seq{10}y^{-1} \qquad y\seq{1} \Leftrightarrow \seq{11}y
\]
\[
y^{-1}\seq{1} \Leftrightarrow \seq{11} y^{-1} \qquad y^{-1}\seq{10} \Leftrightarrow \seq{01} y \qquad y^{-1}\seq{11} \Leftrightarrow \seq{1} y^{-1}.
\]
Two elements of $\tilde T$ are equivalent if the sets of equivalence classes of their elements coincide.
In terms of labeled tree diagrams, this means that $S$ and $T$ are equivalent
if $T$ can be obtained from $S$ by a sequence of substitutions of the following form:
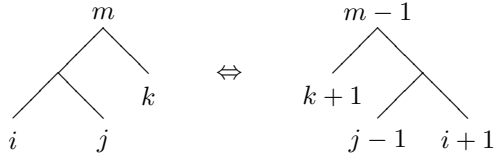
\[
\xy
(0,-6); (6,0)**@{-}; (12,6)**@{-}; (18,0)**@{-}; (6,0); (12,-6)**@{-};
(12,8)*{m};
(0,-9)*{i};
(12,-9)*{j};
(18,-3)*{k};
\endxy
\qquad
\Leftrightarrow
\qquad
\xy
(0,-6); (-6,0)**@{-}; (-12,6)**@{-}; (-18,0)**@{-}; (-6,0); (-12,-6)**@{-};
(-12,8)*{m-1};
(0,-9)*{i+1};
(-12,-9)*{j-1};
(-18,-3)*{k+1};\endxy
\]
In many simple computations, labels are either $0$, $1$ or $-1$.
In this case it is convenient to use $\bullet$ and $\circ$ to indicate the labels
$1$ and $-1$ respectively.
The substitution rule above then becomes a pair of substitutions:
$
\xy
(0,-2); (2,0)**@{-}; (3.5,1.5)**@{-}; (4.5,1.5); (6,0)**@{-}; (2,0); (4,-2)**@{-};
(4,2)*{\bullet};
\endxy
\Leftrightarrow
\xy
(0,-2); (-2,0)**@{-}; (-4,2)**@{-}; (-6,0)**@{-}; (-2,0); (-3.5,-1.5)**@{-};
(0,-2)*{\bullet};
(-4,-2)*{\circ};
(-6,0)*{\bullet};
\endxy
$ and
$
\xy
(0,-2); (-2,0)**@{-}; (-3.5,1.5)**@{-}; (-4.5,1.5); (-6,0)**@{-}; (-2,0); (-4,-2)**@{-};
(-4,2)*{\circ};
\endxy
\Leftrightarrow
\xy
(0.5,-1.5); (2,0)**@{-}; (4,2)**@{-}; (5.5,0.5)**@{-}; (2,0); (3.5,-1.5)**@{-};
(0,-2)*{\circ};
(4,-2)*{\bullet};
(6,0)*{\circ};
\endxy
$.
Notice that if $S$ and $T$ are equivalent elements of $\tilde T$, then $S$ and $T$ have
the same number of leaves.

Equivalence of labeled tree diagrams is generated by the equivalence of trees, together
with the following manipulations on tree diagrams:
\begin{itemize}

\item If $S \to T$ is a labeled tree diagram, then we can insert a caret below the $i\Th$ leaf of $S$
and below the $i\Th$ leaf of $T$ to produce an equivalent diagram $S' \to T'$.
The labels of the top vertexes of the new carets are the same as the original vertexes; the leaves of the
new carets are labeled $0$.

\item If $S \to T$ is a labeled tree diagram, then we may add $1$ to the label of the $i\Th$ leaves of $S$ and of
$T$ to produce an equivalent diagram $S' \to T'$.

\end{itemize}

If $S \to T$ is a labeled tree diagram and $S$ has no labels, then it describes
a continuous function $g:2^\Nbb \to 2^\Nbb$ as follows.
If $\xi$ is an infinite sequence in the alphabet $\{\seq{0},\seq{1},y,y^{-1}\}$ with only finitely many occurrences
of $y$ or $y^{-1}$, then define $\lim \xi$ to be the unique infinite binary sequence $\eta$ such that every prefix of $\eta$
occurs as a prefix of a sequence equivalent to $\xi$.
If $s_i$ and $t_i$ are the $i\Th$-least elements of $S$ and $T$ respectively in the lexicographic order, then define
$g(s_i \xi) = \lim t_i \xi$.
It is easy to check that the generators can be described as follows:
\[
a = (\ 
\xy
(0,-3); (3,0)**@{-}; (6,3)**@{-}; (6,3); (9,0)**@{-}; (3,0); (6,-3)**@{-};
\endxy
\to
\xy
(0,-3); (-3,0)**@{-}; (-6,3)**@{-}; (-6,3); (-9,0)**@{-}; (-3,0); (-6,-3)**@{-};
\endxy
\ )
\qquad
b =
(\ 
\xy
(3,-1.5); (6,1.5)**@{-}; (9,-1.5)**@{-};
(0,1.5); (3,4.5)**@{-}; (6,1.5)**@{-};
(0,-4.5); (3,-1.5)**@{-}; (6,-4.5)**@{-};
\endxy
\to 
\xy
(3,-1.5); (6,1.5)**@{-}; (9,-1.5)**@{-};
(0,1.5); (3,4.5)**@{-}; (6,1.5)**@{-};
(6,-4.5); (9,-1.5)**@{-}; (12,-4.5)**@{-};
\endxy
\ )
\qquad 
c =
(\ 
\xy
(0,-3); (-3,0)**@{-}; (-6,3)**@{-}; (-6,3); (-9,0)**@{-}; (-3,0); (-6,-3)**@{-};
\endxy
\to
\xy
(0,-3); (-3,0)**@{-}; (-6,3)**@{-}; (-6,3); (-9,0)**@{-}; (-3,0); (-6,-3)**@{-}; (-6,-3)*{\bullet}
\endxy
\ )
\]

In fact we can modify this definition slightly in order to associate a function to any labeled tree
diagram: define $g(\lim s_i \xi) = \lim t_i \xi$.
We leave it to the reader to verify that this is a well defined map.
The equivalence of tree diagrams is set up so as to capture exactly
when the corresponding functions coincide.
We will eventually see that the collection of all functions arising in this way is a group which then
coincides with the group $G$ of the previous section.
Notice that if $S \to T$ and $T \to U$ are labeled tree diagrams, then the composition of the two
functions associated to these diagrams is the same as that described by $S \to U$.
In particular $T \to S$ is the inverse of $S \to T$.

We will conclude this section with a illustrative computation.
Notice that $t \mapsto 2t$ correspond to the diagram
$
\xy
(0,-1); (2,1)**@{-}; (4,-1)**@{-};
\endxy
\to
\xy
(0.5,-0.5); (2,1)**@{-}; (3.5,-0.5)**@{-};
(0,-1)*{\circ};
(4,-1)*{\bullet};
\endxy
$.
Conjugating $t \mapsto t+1$ by $t \mapsto 2t$ yields $t \mapsto t+2$,
the square of the first map.
In terms of labeled tree diagrams, this computation can be carried out as follows:
\[
\big(
\xy
(0,-1.5); (3,1.5)**@{-}; (6,-1.5)**@{-};
\endxy
\to 
\xy
(0.5,-1); (3,1.5)**@{-}; (5.5,-1)**@{-}; (0,-1.5)*{\circ}; (6,-1.5)*{\bullet};
\endxy
\big)^{-1}
\cdot 
\big(
\xy
(0,-3); (3,0)**@{-}; (6,3)**@{-}; (9,0)**@{-}; (3,0); (6,-3)**@{-};
\endxy
\to
\xy
(0,0); (3,3)**@{-}; (6,0)**@{-}; (9,-3)**@{-}; (3,-3); (6,0)**@{-};
\endxy
\big)
\cdot 
\big(
\xy
(0,-1.5); (3,1.5)**@{-}; (6,-1.5)**@{-};
\endxy
\to
\xy
(0.5,-1); (3,1.5)**@{-}; (5.5,-1)**@{-}; (0,-1.5)*{\circ}; (6,-1.5)*{\bullet};
\endxy
\big) =\]
\[
\big(
\xy
(-3,-3); (-0.5,-0.5)**@{-}; (0.5,-0.5); (3,-3)**@{-};
(0.5,0.5); (3,3)**@{-}; (5.5,0.5)**@{-}; (0,0)*{\circ}; (6,0)*{\bullet};
\endxy
\to 
\xy
(0,-3); (3,0)**@{-}; (6,3)**@{-}; (9,0)**@{-}; (3,0); (6,-3)**@{-};
\endxy
\big)
\cdot 
\big(
\xy
(0,-3); (3,0)**@{-}; (6,3)**@{-}; (9,0)**@{-}; (3,0); (6,-3)**@{-};
\endxy
\to
\xy
(0,0); (3,3)**@{-}; (6,0)**@{-}; (9,-3)**@{-}; (3,-3); (6,0)**@{-};
\endxy
\big)
\cdot 
\big(
\xy
(0,0); (3,3)**@{-}; (6,0)**@{-}; (9,-3)**@{-}; (3,-3); (6,0)**@{-};
\endxy
\to
\xy
(3,-3); (6,0)**@{-}; (9,-3)**@{-};
(0.5,0.5); (3,3)**@{-}; (5.5,0.5)**@{-}; (0,0)*{\circ}; (6,0)*{\bullet};
\endxy
\big) =\]
\[
\xy
(-3,-3); (-0.5,-0.5)**@{-}; (0.5,-0.5); (3,-3)**@{-};
(0.5,0.5); (3,3)**@{-}; (5.5,0.5)**@{-}; (0,0)*{\circ}; (6,0)*{\bullet};
\endxy
\to 
\xy
(3,-3); (6,0)**@{-}; (9,-3)**@{-};
(0.5,0.5); (3,3)**@{-}; (5.5,0.5)**@{-}; (0,0)*{\circ}; (6,0)*{\bullet};
\endxy
=
\xy
(-3,-1.5); (-0.5,1)**@{-}; (0.5,1); (3,-1.5)**@{-}; (6,-4.5)**@{-}; (3,-1.5); (0,-4.5)**@{-};
(0.5,2); (3,4.5)**@{-}; (5.5,2)**@{-}; (0,1.5)*{\circ}; (6,1.5)*{\bullet};
\endxy
\to 
\xy
(3,-1.5); (6,1.5)**@{-}; (9,-1.5)**@{-};
(0.5,2); (3,4.5)**@{-}; (5.5,2)**@{-};
(0,-4.5); (3,-1.5)**@{-}; (6,-4.5)**@{-};
(0,1.5)*{\circ}; (6,1.5)*{\bullet};
\endxy
=
\xy
(0.5,-4); (3,-1.5)**@{-}; (6,1.5)**@{-}; (9,4.5)**@{-};
(3,-1.5); (5.5,-4)**@{-};
(6,1.5); (8.5,-1)**@{-};
(9,4.5); (11.5,2)**@{-};
(0,-4.5)*{\circ};
(6,-4.5)*{\bullet};
(9,-1.5)*{\circ};
(12,1.5)*{\bullet};
\endxy
\to
\xy
(-0.5,-4); (-3,-1.5)**@{-}; (-6,1.5)**@{-}; (-9,4.5)**@{-};
(-3,-1.5); (-5.5,-4)**@{-};
(-6,1.5); (-8.5,-1)**@{-};
(-9,4.5); (-11.5,2)**@{-};
(0,-4.5)*{\bullet};
(-6,-4.5)*{\circ};
(-9,-1.5)*{\bullet};
(-12,1.5)*{\circ};\endxy
\]
\[
\xy
(0,-4.5); (3,-1.5)**@{-}; (6,1.5)**@{-}; (9,4.5)**@{-};
(3,-1.5); (6,-4.5)**@{-};
(6,1.5); (9,-1.5)**@{-};
(9,4.5); (12,1.5)**@{-};
\endxy
\to
\xy
(0,-4.5); (-3,-1.5)**@{-}; (-6,1.5)**@{-}; (-9,4.5)**@{-};
(-3,-1.5); (-6,-4.5)**@{-};
(-6,1.5); (-9,-1.5)**@{-};
(-9,4.5); (-12,1.5)**@{-};
\endxy
=
\big(
\xy
(0,-3); (3,0)**@{-}; (6,3)**@{-}; (9,0)**@{-}; (3,0); (6,-3)**@{-};
\endxy
\to
\xy
(0,0); (3,3)**@{-}; (6,0)**@{-}; (9,-3)**@{-}; (3,-3); (6,0)**@{-};
\endxy
\big)^2
\]

\section{Sufficiency of the relations}
\label{sufficient}

In this section, we will prove that the relations in
$R$ and $R_0$ are sufficient to give presentations for $G$ and $G_0$.
We will use without proof that the relations in $R$ which only refer to the generators in $X$
give a presentation for $F$ (see \cite{Belk} \cite{Cannon}).
The strategy of the proof is as follows.
First, we will argue that any $S$-word can be put into a \emph{standard form} by applying the
relations.
Standard forms are not unique, but are organized in a way which better facilitates further
symbolic manipulations.
We will then define the notion of a \emph{sufficiently expanded} standard form, argue
that every standard form can be sufficiently expanded by applying the relations in $R$, and
that any sufficiently expanded standard form which
represents an element of $F$ is an $X$-word.

We will begin by defining some terminology.
In what follows, we will say that an $S$-word $\Omega_1$ is \emph{derived from} an $S$-word
$\Omega_0$ if it is the result of applying substitutions of the following forms:
\[
y_t^i x_s^{\pm1} \Rightarrow x_s^{\pm1} y_{t.x_s^{\pm1}}^i  \qquad
y_s \Rightarrow x_s y_{s\seq{0}} y_{s\seq{10}}^{-1} y_{s\seq{11}}
\]
\[
 y_u y_v \Leftrightarrow y_v y_u \qquad x^{i+j} \Leftrightarrow x^i x^j
\qquad y^{i+j} \Leftrightarrow y^i y^j 
\]
\[
\textrm{delete an occurrence of } y^i y^{-i}
\]
where $s,t,u,v \in 2^{<\Nbb}$ are such that $t.x_s$ is defined and $u$ and $v$ are incompatible,
and $i,j$ are nonzero integers of the same sign.
We will write this symbolically as $\Omega_0 \Rightarrow \Omega_1$.
Notice that each of these substitutions corresponds either to a relation in $R$ or to
a group-theoretic identity.
Also observe that only $S_0$-words can be derived from $S_0$-words.

\begin{defn}
An $S$-word $\Omega$ is in \emph{standard form} if
it is the concatenation of a $X$-word followed by a $Y$-word and
whenever $\Omega(i) = y_s^m$, $\Omega(j) = y_t^n$, and $s \subseteq t$,
then $j \leq i$.
We will write \emph{standard form} to mean an $S$-word in standard form.
The \emph{depth} of a standard form $\Omega$
is the least $l$ such that there is binary sequence
$s$ of length $l$ such that $y_s$ occurs in $\Omega$ (if $\Omega$ is an $X$-word, then
we say that $\Omega$ has infinite depth).
\end{defn}

Notice in particular that a given $y_s$ can occur at most once in a standard form (although
possibly with an exponent other than $\pm1$).
Observe that any group element which is expressible by a word in standard form allows us to describe the group
element via a labeled tree diagram in the sense of the previous section.

\begin{lem}\label{base_case}
For every $s \in 2^{<\Nbb}$ and every $l \in \Nbb$, there is a standard form
$\Omega$ which can be derived from $y_s^{\pm1}$ such that:
\begin{enumerate}

\item \label{x_local}
if $x_u$ occurs in $\Omega$, then $u$ extends $s$;

\item if $y_u$ occurs in $\Omega$, then $u$ extends $s$, has length at least $l$,
 and the exponent of $y_u$ is $\pm 1$;

\item if $y_u$ and $y_v$ occur in $\Omega$ and $u \ne v$, then $u$ and $v$ are
incompatible.

\end{enumerate}
\end{lem}

\begin{proof}
The proof is by induction on $l-|s|$.
If $l -|s| = 0$, there is nothing to do since $y_s$ already satisfies the conclusion of the lemma.
If $l -|s| > 0$, then $y_s \Rightarrow x_s y_{s\seq{0}} y_{\seq{10}}^{-1} y_{\seq{11}}$ and we can apply
the induction hypothesis to obtain $\Omega_{s\seq{0}}$, $\Omega_{s\seq{10}}$, $\Omega_{s\seq{11}}$
which satisfy the conclusion of the lemma for $y_{s\seq{0}}$, $y_{s\seq{10}}^{-1}$, and
$y_{s\seq{11}}$ respectively for the same value of $l$.
By conclusion \ref{x_local} of the lemma, we can apply substitutions of the form
$y_v x_u \Rightarrow x_u y_v$ for incompatible $u$ and $v$ move the occurrence of $x_u$ in
\[
x_s \Omega_{s \seq{0}} \Omega_{s \seq{10}} \Omega_{s \seq{11}}
\]
to the left, placing the word a standard form which satisfies the conclusions of the lemma. 
The case of $y_s^{-1}$ is handled similarly using the substitution 
$y_s^{-1} \Rightarrow x_s^{-1} y_{s\seq{00}}^{-1} y_{s \seq{01}} y_{s \seq{1}}^{-1}$.
\end{proof}

\begin{lem} \label{YXtoXY}
If $\Xi$ is an $X$-word, then there is an $l_0$ such that if $\Omega$ is a standard
form of depth $l \geq l_0$, then $\Omega \Xi \Rightarrow \Omega'$ for some standard form $\Omega'$
of depth at least $l-k$ where $k$ is the word length of $\Xi$. 
\end{lem}

\begin{proof}
If $\Xi = x_s^{\pm1}$, then observe that if $t$ is a finite binary sequence of length at least $l+2$,
then  $t . x_s^{\pm1}$ is defined and its length differs from
the length of $\Omega$ by at most $1$.
Thus by repeated applications of the substitution $y_t^i x_s^{\pm1} \Rightarrow x_s^{\pm1} y^i_{t.x_s^{\pm1}}$,
the final occurrence
of $x_s$ in $\Omega x_s^{\pm1}$ can be moved to the left of all occurrences of a $y_t$.
This results in a standard form in which the depth is changed by at most $1$.
The general case now follows by induction.
\end{proof}

\begin{lem}\label{stdform}
If $\Omega$ is any $S$-word and $l \in \Nbb$, then $\Omega \Rightarrow \Omega'$ for some standard form
$\Omega'$ of depth at least $l$.
\end{lem}

\begin{proof}
The proof is by double induction: first on the word length $n$ of $\Omega$ and then on $l$.
The case $n=1$ is handled by Lemma \ref{base_case}.
By making a substitution of the form $a^{\pm(k+1)} \Rightarrow a^{\pm k} a^{\pm1}$ if necessary, we may assume
that $\Omega = \Omega_0 \Omega_1$ where $\Omega_i$ is an $S$-word of positive length.
By our induction hypothesis $\Omega_1 \Rightarrow \Xi \Upsilon$,
where $\Xi$ and $\Upsilon$ are $X$- and $Y$-words respectively and $\Upsilon$ has depth $l$.
Let $k$ be the word length of $\Xi$ and let $m \geq l$ be such that if $y_u$ occurs in $\Upsilon$,
$u$ has length less than $m$.
By our induction hypothesis, there is a standard form $\Omega_0'$ of depth at least $m+k$ such
that $\Omega_0 \Rightarrow \Omega_0'$.
By Lemma \ref{YXtoXY} we have that $\Omega_0' \Xi \Rightarrow \Omega_0''$ for some standard form
$\Omega_0''$ of depth at least $m$,
we have:
\[
\Omega \Rightarrow \Omega_0 \Omega_1 \Rightarrow \Omega_0 \Xi \Upsilon \Rightarrow
\Omega_0' \Xi  \Upsilon \Rightarrow \Omega_0'' \Upsilon
\]
Finally, notice that since the depth of $\Omega_0''$ is at least $m$, $\Omega' = \Omega_0'' \Upsilon$
is a standard form of depth at least $l$, as desired.
\end{proof}

If $\Omega$ is standard form and $y_s$ occurs in $\Omega$, we say that $s$
is \emph{exposed in $\Omega$} if there is a finite binary sequence $u$ extending
$s$ such that if $t$ is a binary sequence compatible with $u$ and $y_t$ occurs in $\Omega$,
then $t$ is an initial part of $s$.

\begin{defn}
A standard form $\Omega$ is \emph{sufficiently expanded} if whenever $y_s$ occurs in $\Omega$
and $s$ is not exposed in $\Omega$, then:
\begin{itemize}

\item $y_{s \seq{0}}$ occurs in $\Omega$ if $y_s$ occurs positively in $\Omega$;

\item $y_{s \seq{1}}$ occurs in $\Omega$ if $y_s$ occurs negatively in $\Omega$.

\end{itemize}
\end{defn}

The motivation for this definition is as follows.
Suppose that $\Omega$ is a standard form which is not sufficiently expanded and that
this is witnessed by $\Omega(i) = y_s^n$ for $n > 0$.
If we substitute
\[
x_s y_{s \seq{0}} y_{s \seq{10}}^{-1} y_{s \seq{11}} y_s^{n-1}
\]
for $y_s^n$ in $\Omega$, then whenever $y_t$ occurs before $y_s$
in $\Omega$, $t . x_s$ is defined.
A similar conclusion holds --- with $x_s^{-1}$ replacing $x_s$ --- if $n < 0$ and the substitution
\[
x_s^{-1} y_{s \seq{00}}^{-1} y_{s \seq{01}} y_{s \seq{1}}^{-1} y_s^{n+1}
\]
is applied.
This plays an important role in the proof of the next lemma.

\begin{lem}\label{expanded}
If $\Omega$ is a standard form, then there is a sufficiently expanded standard form
which can be derived from $\Omega$.
\end{lem}

\begin{proof}
We will prove the lemma by defining a \emph{well-founded} partial ordering $\triord$ on the set of standard forms 
and a notion of expansion on standard forms which are not sufficiently expanded in such a way
that produces a smaller standard form in this ordering.
Here a partial order is \emph{well-founded} if it has no infinite strictly decreasing sequences.
We will define the ordering first.

If $\Omega$ is an $S$-word, let $T(\Omega)$ denote the minimal prefix set which
has the property that if $y_t$ occurs in $\Omega$, then $t$ has an extension in $T$.
If $\Omega_0$ and $\Omega_1$ are standard forms, define $\Omega_0 \triord \Omega_1$ if
$|T(\Omega_0)| < |T(\Omega_1)|$ or $|T(\Omega_0)| = |T(\Omega_1)|$ and $|k_0| < |k_1|$ where
$k_i$ is the exponent in $\Omega_i$ of the $\leq_\lex$-maximal $s$ such that $y_s$ occurs in at least
one of $\Omega_0$ or $\Omega_1$ and for which $k_0 \ne k_1$
(if $y_s$ occurs in only one of the $\Omega$'s, then the other exponent is
$0$).
Notice that for a fixed $m$ there are only finitely many prefix sets of cardinality $m$.
In particular the collection $F$ of all finite binary sequences which have an extension in a prefix set
of cardinality $m$ is finite.
Since the lexicographic ordering on $\Nbb^F$ is a well-order, $\triord$ is well-founded.

Now suppose that $\Omega$ is a standard form which
is not sufficiently expanded as witnessed by $\Omega(i) = y_s^n$.
For simplicity, suppose that $n > 0$ and apply the substitution
\[
y_s^n \Rightarrow x_s y_{s \seq{0}} y_{s \seq{10}}^{-1} y_{s \seq{11}} y_s^{n-1}
\]
(if $n=1$, the $y_s^{n-1}$ term is omitted)
followed by substitutions of the form $y_t^m x_s \Rightarrow x_s y_{t.x_s}^m$ to move
$x_t$ to the left, forming a new word $\Omega'$ which is the concatenation
of a $X$-word followed by a $Y$-word.
At this point, the only thing preventing $\Omega'$
from being a standard form is the newly introduced occurrences of 
$y_{s \seq{0}}$, $y_{s \seq{10}}$, and $y_{s \seq{11}}$.
Observe that $y_{s \seq{1}}$ can not occur in $\Omega'$;
for this to happen, $s \seq{1}$ would have to equal $t.x_s$
for some $t$ such that $t.x_s$ is defined, and such a $t$ does not exist.
Furthermore, if $y_t$ occurs in $\Omega'$ and $t$ properly extends one of the sequences
$s \seq{0}$, $s \seq{10}$, or $s \seq{11}$, then $y_t$ must occur before any occurrence
of $y_{s \seq{0}}$, $y_{s \seq{10}}$, or $y_{s \seq{11}}$ in $\Omega'$.
Similarly, if $t$ is a proper initial part of $s \seq{0}$, $s \seq{10}$, and $s \seq{11}$ and $y_t$ occurs
in $\Omega'$, then $t$ is actually an initial part of $s$ and thus the occurrence is after the point of the
substitution.
We may therefore apply substitution of the form $y_u y_v \Leftrightarrow y_v y_u$
for incompatible $u$ and $v$
in order to move any two distinct occurrences
of $y_{s \seq{0}}$, $y_{s \seq{10}}$, or $y_{s \seq{11}}$ to the same position in $\Omega'$, resulting
in a word $\Omega''$ which is now in standard form.

It now suffices to show that $\Omega'' \triord \Omega$.
Since $s$ was not exposed in $\Omega$, each of $s \seq{00}$, $s \seq{01}$, and $s \seq{1}$ has
an extension in $T(\Omega)$; recall that, by assumption, $y_{s \seq{0}}$ does not occur in $\Omega$.
It follows that $t. x_s$ is defined for every element $t$ of $T(\Omega)$ and that
$T(\Omega') = \{t . x_s  : t \in T(\Omega)\}$.
Hence $T(\Omega)$ and $T(\Omega')$ have the same cardinality.
Notice that if $y_s$ occurs in $\Omega''$, it occurs in $\Omega'$ and hence
$T(\Omega')$ dominates $T(\Omega'')$.
It follows
that $T(\Omega'')$ has cardinality at most that of $T(\Omega)$.
If $T(\Omega'')$ has the same cardinality as $T(\Omega)$, then $s$ is the $\leq_\lex$-maximal sequence
such that the exponent of $y_s$ in $\Omega$ and $\Omega''$ differs and in this 
case, it decreases by one in absolute value.
Thus we have shown $\Omega'' \triord \Omega$.
\end{proof}

Let $B$ denote the set $\{\seq{0},\seq{1},y,y^{-1}\}$ and let $B^{<\Nbb}$ denote
the collection of all finite strings of elements of $B$.
If $\Lambda$ is in $B^{<\Nbb}$ and $\Lambda(i)$ is either $y$ or $y^{-1}$, we will
say that $\Lambda(i)$ is an \emph{occurrence of $y^\pm$}.
We will use $B$-words to analyze the evaluation of standard forms at binary sequences.
The following symbolic manipulations correspond to the recursive definition
of the function $y:2^{\Nbb} \to 2^{\Nbb}$.

\begin{defn}
Suppose that $\Lambda$ is in $B^{<\Nbb}$.
An application of one of the substitutions
\[
y\seq{00} \Rightarrow \seq{0}y 
\qquad
y\seq{01} \Rightarrow \seq{10}y^{-1}
\qquad
y\seq{1} \Rightarrow \seq{11}y
\]
\[
y^{-1} \seq{0} \Rightarrow \seq{00}y^{-1}
\qquad
y^{-1} \seq{10} \Rightarrow \seq{01}y
\qquad
y^{-1} \seq{11} \Rightarrow \seq{1}y^{-1}
\]
at an occurrence of $y^\pm$ is said to \emph{advance} that symbol.
If several advances of occurrences of $y^{\pm}$ are applied to $\Lambda$,
resulting in $\Lambda'$,
then we say that $\Lambda$ \emph{can be advanced to} $\Lambda'$, denoted
$\Lambda \Rightarrow \Lambda'$.
\end{defn}

\begin{defn}
Suppose that $\Lambda$ is in $B^{<\Nbb}$.
An occurrence of $y^\pm$ is a \emph{potential cancellation}
in $\Lambda$ if repeatedly advancing it results in an occurrence of the substring
$y y^{-1}$ or $y^{-1} y$ in the modified word.
\end{defn}

\begin{lem} \label{potential_cancelation}
Suppose that $\Lambda$ is in $B^{<\Nbb}$ and
contains no potential cancellations.
Then advancing any occurrence of a $y^{\pm}$ results in a word with no
potential cancellations.
\end{lem}

\begin{proof}
Suppose that $\Lambda$ is given and that the $i$th occurrence of $y^{\pm}$ is advanced to create $\Lambda'$.
The only possibility for introducing a potential cancellation is if $i > 1$ and a potential cancellation
occurs at the $i-1$st occurrence of a $y^{\pm}$ in $\Lambda'$.
Return to $\Lambda$ and advance the $i-1$st occurrence of $y^{\pm}$ as much as possible to produce $\Lambda''$.
Suppose for a moment that after advancing, this occurrence is a $y$; notice that the next symbol is either
$\seq{0}$ or $y$.
We now have the following cases:
\[
yy \seq{00} \Rightarrow y \seq{0} y
\]
\[
yy \seq{01} \Rightarrow y \seq{10} y^{-1} \Rightarrow \seq{11} y \seq{0} y^{-1}
\]
\[
yy \seq{1} \Rightarrow y \seq{11} y \Rightarrow \seq{1111} yy
\]
\[
y \seq{0} y \seq{00} \Rightarrow y \seq{00} y \Rightarrow 0 yy
\]
\[
y \seq{0} y \seq{01} \Rightarrow y \seq{010} y^{-1} \Rightarrow \seq{10} y^{-1} 0 y^{-1} \Rightarrow \seq{1000} y^{-1} y^{-1}
\]
\[
y \seq{0} y \seq{1} \Rightarrow  y \seq{011} y \Rightarrow \seq{10} y^{-1} \seq{1} y
\]

\[
y \seq{0} y^{-1} \seq{0} \Rightarrow y \seq{000} y^{-1} \Rightarrow \seq{0} y \seq{0} y^{-1}
\]
\[
y \seq{0} y^{-1} \seq{10} \Rightarrow y \seq{001} y \Rightarrow \seq{0} y \seq{1} y \Rightarrow \seq{011} y y
\]
\[
y \seq{0} y^{-1} \seq{11} \Rightarrow  y \seq{01} y^{-1} \Rightarrow \seq{10} y^{-1} y^{-1}
\]
The above lines list the possible contexts for two occurrences of $y^\pm$ in the $i-1$st and $i$th in $\Lambda''$
where the first occurrence is positive.
In the above cases, the $i$th occurrence of $y^\pm$ is advanced in $\Lambda''$
and then the $i-1$st occurrence is advanced as much as possible, demonstrating that a cancellation
does not occur.
The case in which the $i-1$st occurrence of $y^\pm$ in $\Lambda''$ is $y^{-1}$
is handled by symmetry ---
the rules for advancement and potential cancellation are invariant under the following
involution:
\[
y \Leftrightarrow y^{-1}
\qquad
\seq{0} \Leftrightarrow \seq{1}.
\]
\end{proof}

\begin{lem} \label{advance_yn}
Suppose that $\Lambda$ is in $B^{<\Nbb}$ and contains no potential cancellations.
There is a finite binary sequence $u$ such that $\Lambda \cat u$ can be advanced
to $s \cat y^{n}$ for some binary sequence $s$,
where $n$ is the number of occurrences of $y^\pm$ in $\Lambda$.
\end{lem}

\begin{proof}
The proof is by induction on the number of occurrences of $y^\pm$ in $\Lambda$.
If there is only one occurrence, advance the occurrence as many times as possible,
resulting in $s \cat y$, $s \cat y^{-1}$, $s \cat y \cat \seq{0}$, or $s \cat y^{-1} \cat \seq{1}$
for some finite binary sequence $s$.
In the first case we are finished; in the
remaining cases, the choices $u = \seq{10}$, $u=\seq{0}$, and $u = \seq{0}$ work.
Now suppose that $\Lambda$ contains $n+1$ occurrences of $y^\pm$.
Induction and Lemma \ref{potential_cancelation} reduce the general case to
the two special cases
$y \seq{0} y^n$ and $y^{-1} \seq{1} y^n$.
In these cases, use $u = \seq{0}^{2^n}$, observing that $y^n \seq{0}^{2^n}$ can be advanced
to $\seq{0} y^n$.
\end{proof}

\begin{lem}\label{evaluation}
If $\Omega$ is a sufficiently expanded standard form
then either $\Omega$ is an $X$-word or else
$\Omega$ does not have the same evaluation as an $X$-word.
\end{lem}

\begin{proof}
Notice that it is sufficient to prove the lemma when $\Omega$ is a sufficiently expanded standard form
which is a $Y$-word of positive length.
Suppose that such an $\Omega$ is given and let $g:2^\Nbb \to 2^\Nbb$ be the evaluation of
$\Omega$ in $G$.
It will be sufficient to construct finite binary sequences $u$ and $v$ such that
the value of $g$ at $u \cat \xi$ is
$v \cat y^n (\xi)$ for some $n > 0$.
This is because if $\xi = \seq{0}^{2^n} \seq{1} \seq{0}^{2^n} \seq{1} \ldots$, then
the value of $g$ at $u \cat \xi$ is $v \cat \seq{0}\seq{1}^{2^n}\seq{0}\seq{1^{2^n}} \ldots$, which is not tail
equivalent to $u \cat \xi$.

The finite binary sequence $u$ will be constructed by a recursive procedure.
Let $u \restriction i_0$ be the finite binary sequence
such that the last entry of $\Omega$ is a power of $y_{u \restriction i_0}$.
Suppose that $u \restriction i$ has been defined and that $y_{u \restriction i}$ occurs
in $\Omega$.
If $u \restriction i$ is exposed in $\Omega$, then let $u \restriction l$
be any finite binary sequence extending $u \restriction i$ which witnesses this.
Otherwise, define $u(i) = \seq{0}$ if $y_{u \restriction i}$ occurs positively in $\Omega$ and
$u(i) = \seq{1}$ if $y_{u \restriction i}$ occurs negatively in $\Omega$.
Define $\Lambda$ to be the result of simultaneously inserting $y^n$ after 
$s \restriction i$ whenever $y_{s \restriction i}^n$ occurs in $\Omega$.
Notice that by the choice of our sequence $u$,
$\Lambda$ does not contain potential cancellations: any occurrence of $y$ except for the final occurrence of $y^{\pm}$,
is
followed by $\seq{0} y^\pm$ and any occurrence of $y^{-1}$ except for the final occurrence of $y^{\pm}$ is followed by
$\seq{1} y^{\pm}$.
It follows from Lemma \ref{advance_yn} that there is a sequence $s$ such that
$\Lambda \cat s$ can be advanced to $v y^n$ for some binary sequence $v$,
where $n$ is the number of occurrences of $y^{\pm}$ in $\Lambda$
(this number coincides with the number of steps of the recursive procedure above, which is at least 1).
Set $u$ to be the concatenation of $u \restriction l$ followed by $s$.

Recall now that we have assumed that $\Omega$ is a $Y$-word;
$g = y_{t_k}^{n_k} \ldots y_{t_1}^{n_1}$.
\[
\xi . g = \xi . (y_{t_1}^{n_1} \cdots y_{t_k}^{n_k}) = (\cdots (\xi . y_{t_1}^{n_1}). \cdots ). y_{t_k}^{n_k}
\]
Let $\xi_i$ be the result of applying the $y_{t_1}^{n_1} \cdots y_{t_{i-1}}^{n_{i-1}}$ to $\xi$.
Observe that if $t_{i+1}$ is an initial part of $\xi$, then it is still an initial part of $\xi_i$.
This follows from the fact that if $j \leq i$, then $t_j$ either extends $t_{i+1}$ or else is incompatible with
$t_{i+1}$.
In particular, if $t_{i+1}$ is not an initial part of $\xi$, then $\xi_{i+1} = \xi_i$.
If $t_{i+1}$ is an initial part of $\xi_i$, then $\xi_{i+1} = t_{i+1} \cat y^{n_i} (\eta_i)$, where
$\xi_i = t_i \cat \eta_i$.
It follows from $\Lambda \cat s \Rightarrow v y^n$ that $\xi.g = v \cat (\eta.y^n)$, where
$\eta$ is such that $\xi = u \cat \eta$.
\end{proof}

To see that this finishes the proof of the main theorem, 
suppose that $\Omega$ is an $S$-word which evaluates to the identity function in $G$.
By Lemma \ref{stdform}, $\Omega \Rightarrow \Omega'$ for some word $\Omega'$ in standard form. 
By Lemma \ref{expanded}, $\Omega' \Rightarrow \Omega''$ for some word $\Omega''$
which is in standard form and which is sufficiently expanded. 
In particular, $\Omega''$ is equivalent to $\Omega$ by the relations in $R$;
if $\Omega$ was an $S_0$-word, then $\Omega''$ is an $S_0$ word 
and the derivation $\Omega \Rightarrow \Omega' \Rightarrow \Omega''$ 
utilizes only relations in $R_0$. 
By Lemma \ref{evaluation}, $\Omega''$ is an $X$-word. 
Since $R$ includes a presentation for $F$, $\Omega''$ 
can be reduced to the identity using the relations in $R_0$.

\end{document}

%% file: macros.tex
\newtheorem{thm}{Theorem}[section]

\newtheorem{lem}[thm]{Lemma}
\newtheorem*{lem*}{Lemma}

\newtheorem{prop}[thm]{Proposition}

\theoremstyle{remark}
\newtheorem{remark}[thm]{Remark}

\theoremstyle{definition}
\newtheorem{defn}[thm]{Definition}

\newcounter{my_enumerate_counter}

\newcommand\comment[1]{}

\newcommand\PSL{\operatorname{PSL}}

\newcommand\Tcal{\mathcal{T}}

\newcommand\Nbb{\mathbb{N}}

\newcommand\Qbb{\mathbb{Q}}
\newcommand\Rbb{\mathbb{R}}

\newcommand\Zbb{\mathbb{Z}}

\newcommand{\seq}[1]{\mathtt{#1}}
\newcommand\cat{{}^\smallfrown}

\newcommand{\triord}{\triangleleft}

\newcommand{\Th}{{}^{\textrm{th}}}

\newcommand\lex{\mathrm{lex}}

\newcommand{\Seq}[1]{\langle #1 \rangle}

\renewcommand{\epsilon}{\varepsilon}